\documentclass[journal,twoside,web]{ieeeconf}

\usepackage{cite}

\usepackage{amsmath,amsfonts,amssymb}
\usepackage{cleveref}
\usepackage{graphicx,xcolor}
\usepackage{comment}
\newcommand{\R}{\mathbb{R}}
\newcommand{\N}{\mathbb{N}}
\DeclareMathOperator*{\argmin}{arg\,min}

\newtheorem{thm}{Theorem}

\newtheorem{prop}{Proposition}
\newtheorem{cor}{Corollary}

\newtheorem{defn}{Definition}
\newtheorem{ass}{Assumption}
\crefname{ass}{Assumption}{Assumptions}
\crefname{prop}{Proposition}{Propositions}

\title{\LARGE \bf
A model-free approach to control barrier functions using funnel control
}

\author{Lukas Lanza$^{\star,1}$, Johannes Köhler$^{\star,2}$, Dario Dennstädt$^{1,3}$, Thomas Berger$^{3}$, and Karl Worthmann$^{1}$%
\thanks{$\star$The first two authors contributed equally to this work}
\thanks{$^{1}$Optimization-based Control Group, TU Ilmenau, Ilmenau, Germany, {\tt\small $\{$lukas.lanza, karl.worthmann$\}$@tu-ilmenau.de}}%
\thanks{$^{2}$ Institute for Dynamic Systems and Control, ETH Zurich, 8052 Zürich, Switzerland,
         {\tt\small jkoehle@ethz.ch}}%
\thanks{$^{3}$Universit\"at Paderborn, Institut f\"ur Mathematik, Warburger Str.~100, 33098~Paderborn, Germany,  {\tt\small thomas.berger@math.upb.de,  dario.dennstaedt@uni-paderborn.de}}%
\thanks{Funding: T.~Berger, D.~Dennstädt, L.~Lanza and K.~Worthmann gratefully acknowledge funding by the Deutsche Forschungsgemeinschaft (DFG, German Research Foundation)~-- Project-ID 471539468.}%
}

\begin{document}

\maketitle

\begin{abstract}
Control barrier functions (CBFs) are a popular approach to design feedback laws that achieve safety guarantees for nonlinear systems.
The CBF-based controller design relies on the availability of a model to select feasible inputs from the set of CBF-based controls. 
In this paper, we develop a model-free approach to design CBF-based control laws, eliminating the need for knowledge of system dynamics or parameters.

Specifically, we address safety requirements characterized by a time-varying distance to a reference trajectory in the output space and construct a CBF that depends only on the measured output.
Utilizing this particular CBF, we determine a subset of CBF-based controls without relying on a model of the dynamics by using techniques from funnel control. 
The latter is a model-free high-gain adaptive control methodology, which achieves tracking guarantees via reactive feedback.
In this paper, we discover and establish a connection between the modular controller synthesis via zeroing CBFs and model-free reactive feedback.
The theoretical results are illustrated by a numerical simulation.
\end{abstract}


\section{Introduction} \label{Sec:Introduction}
A key requirement in many control applications is the assurance of safety-critical constraints during runtime. 
This task can be quite challenging if the system exhibits complex and nonlinear dynamics. 
In recent years, a broad set of control tools has been developed to address this issue, see~\cite{wabersich2023data} for an overview. 
However, application of these methods typically hinges on having an accurate model of the system dynamics and implementation can be challenging for complex high-dimensional models. 
In contrast, it is difficult to obtain safety guarantees for popular
model-free control techniques such as PID control, active disturbance rejection control~\cite{han2009pid}, or model-free reinforcement learning methods~\cite{brunke2022safe}.
This paper fuses techniques from control barrier functions and funnel control to provide a control methodology that ensures safe operation for an entire class of nonlinear systems without relying on a model of the underlying dynamics. Hence, our approach is model-free.

\subsubsection*{Related work}
\textit{Control barrier functions (CBFs)}, first introduced in~\cite{wieland2007constructive}, can ensure satisfaction of safety-critical constraints and they have gained popularity in robotics applications because of their simplicity and modularity, see~\cite{ames2016control,ames2019control} for an overview. 
Specifically, CBFs determine the control input by solving a quadratic program (QP) that enforces a desired bound on the derivative of the barrier, which then ensures positive invariance and asymptotic stability of a safe set. However, application of CBF-based control requires a model that describes the nonlinear dynamics, measurement of the state, and implementation complexity scales with the dimension of the nonlinear model. 
The first issues have been addressed in the literature by investigating robustness to bounded disturbances~\cite{jankovic2018robust,gurriet2018towards}, sector nonlinearities~\cite{buch2021robust}, or state estimation error~\cite{dean2021guaranteeing}, and by analyzing input-to-state safety~\cite{alan2023control,kolathaya2018input}. Furthermore, some approaches investigate extracting the required model information from data  \cite{taylor2020learning,dhiman2021control}, which, however, also increases the complexity of application. 
To mitigate complexity, \cite{cohen2024safety} investigate the utilization of simplified reduced-order models. 
Furthermore, for a special class of reduced-order kinematic models, \cite{molnar2021model} provides a \textit{model-free} CBF-based control law, which requires minimal knowledge about the system dynamics. 
In this paper, we develop a new model-free CBF-based approach by leveraging funnel control.

\textit{Funnel control (FC)} and prescribed performance control (PPC) are two distinct but increasingly interconnected methodologies for enforcing predefined bounds on the tracking error to a given reference
in control systems. 
In contrast to CBF-based control, both control methodologies are model-free in the following sense. While the systems to be controlled must satisfy structural requirements like a well-defined relative degree and the same number of inputs and outputs, no knowledge of the system parameters (not even estimates of parameters) is used in the respective feedback laws. Thus, the same feedback law can be used for an entire class of systems.
FC enforces time-varying performance bounds via error-dependent gains, ensuring that the tracking error remains within a given performance funnel over time~\cite{IlchRyan02b,berger2018funnel,BergIlch21}. PPC utilizes algebraic transformations of the tracking error to enforce prescribed transient and steady-state constraints~\cite{bechlioulis2008robust,kostarigka2011adaptive,bechlioulis2014low}.\\
Recently, the relation between \textit{reciprocal} CBFs and PPC was considered in~\cite{namerikawa2024equivalence}.
It has been shown that the auxiliary error variables utilized in PPC can be used to define a reciprocal CBF, and conversely, that the PPC feedback law can be derived from the reciprocal CBF.
In the present work, we provide a new \textit{model-free} CBF-based control law, which can be applied to a large class of nonlinear systems by leveraging connections to funnel control. This is closely related to~\cite{namerikawa2024equivalence}, see the discussion in \Cref{Sec:Discussion} for details.
However, to avoid singularities in the barrier function, we use zeroing CBFs instead of reciprocal CBFs in the present work.

\subsubsection*{Contribution}
We show in \Cref{Sec:MainResults} that the difference between a time-varying predefined boundary and the norm of the tracking error defines a zeroing CBF for the entire class of systems under consideration.
Then, we derive a set of control inputs, which guarantees that the system evolves within a predefined safe set around a time-varying reference, i.e., the inputs contained in that control set render the safe set forward invariant.
Unlike well-known CBF-based controller designs, the CBF-based controller developed in this work is \textit{model-free}.
This means that no system data or parameter estimates are used in the design and therefore the same feedback law can be applied to an entire class of systems. 
Thus, we establish a connection between the modular controller synthesis via CBFs, which relies on the availability of a model, and the model-free funnel control methodology, which guarantees prescribed tracking performance via pure reactive feedback.

\ \\
\textbf{Nomenclature.}
For $x,y \in \R^n$, we use $\langle x, y \rangle = x ^\top y$ and the Euclidean norm is denoted by $\|x\| = \sqrt{\langle x,x \rangle}$. 
For $V \subseteq \R^m$, we denote by $C^k(V,\R^n)$ the set of $k$-times continuously differentiable functions~$f: V \to \R^n$.
For~$f:\R^m \to \R^n$,~$g:\R^k \to \R^m$ and~$x\in\R^k$ we use $(f \circ g)(x) := f(g(x))$.
For an interval~$I \subseteq \R$, $L^\infty(I; \R^{n})$ is the Lebesgue space of measurable essentially bounded functions $f : I\to\R^n$ with norm $\|f \|_{\infty} = \text{esssup}_{t \in I} \|f(t)\|$; for~$k \in \N$, $W^{k,\infty}(I;\R^{n})$ is the Sobolev space of all functions $f:I\to\R^n$ with $k$-th order weak derivative $f^{(k)}$ and $f,f^{(1)},\ldots,f^{(k)}\in L^\infty(I, \R^{n})$.
$\mathcal{K}_\infty^e := \{ \alpha \in C(\R,\R) \, | \, \alpha(0) = 0, \alpha\ \text{mon.\ incr.\ }, \lim_{a\to \pm \infty}\alpha(a) = \pm \infty \}$.
For $b: \R \times \R^n \to \R$ differentiable, $(L_f b)(\cdot)$ denotes the Lie derivative of~$b$ along a vector field~$f$, i.e., $\langle \nabla_\zeta b(t,\zeta), f(\zeta) \rangle$.
Further, we use $\partial_t b(t,\zeta) := \tfrac{\partial}{\partial t}b(t,\zeta)$.

\section{System class and control objective}
In this section, we introduce the class of dynamical control systems and present the control objective. 

{\textbf{System class.}}
We consider nonlinear multi-input multi-output control systems
\begin{equation} \label{eq:System}
\begin{aligned}
    \dot x(t) &= F(x(t)) + G(x(t))u , \quad x(0) = x^0 \in \R^n, \\
    y(t) &= H(x(t)),
\end{aligned}
\end{equation}
where~$x(t) \in \R^n$ is the state, $y(t) \in \R^m$ is the (measured) output, and $u(t) \in \R^m$ is the control input at time ${t\geq 0}$.
Note that the dimensions of the output and input coincide, and for the state dimension we assume $n \ge m$.
The functions $F \in C^1(\R^n;\R^n)$, $H\in C^2(\R^n;\R^m)$, and $G \in C^1(\R^n; \R^{n \times m})$ are unknown, but assumed to satisfy the following structural conditions.
\begin{ass} \label{ass:StructuralAssumptions}
The following should be satisfied by the right-hand side of~\eqref{eq:System}:
    \begin{enumerate}
        \item[a)] the matrix $(L_G H)(x) \in \R^{m \times m}$ is positive definite, i.e., $\langle z,  (L_GH)(x) z \rangle > 0$ for all~$x \in \R^n$ and all ${z \in \R^m \setminus \{0\}}$; 
        \item[b)] there exists a diffeomorphism $T: \R^n \to \R^n$, $x \mapsto (y,\eta)$ such that~\eqref{eq:System} can equivalently be written in normal form
        \begin{equation} \label{eq:System_InputOutput}
            \begin{aligned}
                \dot y(t) &= f(y(t),\eta(t)) + g(y(t),\eta(t)) u(t), \\
             \dot \eta(t) & = q(y(t),\eta(t)),
            \end{aligned}
            \end{equation}
            where we have $f(y,\eta) := (L_F H)(T^{-1}(y,\eta))$, $g(y,\eta) := (L_G H)(T^{-1}(y,\eta))$, and~$\eta(t) \in \R^{n-m}$ is the unmeasured internal state at time~$t \ge 0$;
        \item[c)] the internal dynamics are bounded-input bounded-state stable, i.e., for all~$c_0 > 0$ there exists~$\bar q > 0$ such that for all $\eta^0 \in \R^{n-m}$ and all $\zeta\in L^\infty(\R_{\ge 0},\R^{m})$ it holds that $\|\eta^0\|+ \|\zeta\|_\infty  \leq c_0$ implies ${\|\eta (\cdot;0,\eta^0,\zeta)\|_\infty \leq \bar q}$, where $\eta (\cdot;0,\eta^0,\zeta):\R_{\ge 0}\to\R^{n-m}$ denotes the unique global solution of the second of equations~\eqref{eq:System_InputOutput} for~${y=\zeta}$.\footnote{The aforementioned assumption ensures that the maximal solution $\eta (\cdot;0,\eta^0,\zeta)$ can be extended to a global solution, cf.~\cite[\S~10, Thm.~XX]{Walt98}.}
    \end{enumerate}
\end{ass}
\Cref{ass:StructuralAssumptions}a) means in particular that the systems under consideration have global relative degree one, cf.~\cite[Def.~5.1]{byrnes1991asymptotic}.
Although we do not assume symmetry, we note that for many mechanical systems, $(L_GH)(x)$ simplifies to an inertia-related matrix, which is inherently symmetric and positive-definite and thus satisfies \Cref{ass:StructuralAssumptions}a), cf.~\cite{berger2021tracking}.
While assuming a well-defined relative degree is common in high-gain adaptive control, cf.~\cite{IlchRyan02b,Hack17}, we restrict ourselves to systems with a relative degree one for simplicity of presentation of the main ideas.
A generalization of the results presented here will be topic of future work.
Instead of positive definiteness of~$g(\cdot)$, negative definiteness is a viable assumption as well, which will only change the sign in the feedback law derived later.
We emphasize that, while existence of the diffeomorphism $T$ 
putting system~\eqref{eq:System} in the form~\eqref{eq:System_InputOutput} is assumed in \Cref{ass:StructuralAssumptions}b), concrete knowledge about this transformation is not required -- it is merely used as a tool in the following analysis.
Sufficient conditions for its existence can be found in \cite[Cor.~5.6]{byrnes1991asymptotic}.
\Cref{ass:StructuralAssumptions}c) ensures that the unmeasured internal state evolves within a (unknown) compact set for bounded~$y$.
In what follows, we call a (locally) absolutely continuous function $x: [0, \omega) \to \R^n$, $\omega\in (0,\infty]$, with $x(0) =x^0$ a solution (in the sense of Carathéodory) to~\eqref{eq:System}, if it satisfies~\eqref{eq:System} for almost all~$t \in [0,\omega)$. A solution~$x$ is said to be maximal, if it has no right extension that is also a solution; a maximal solution is global, if $\omega=\infty$.	

\textbf{Control objective.}
We aim at designing a CBF-based feedback law such that the output~$y$ of system~\eqref{eq:System} tracks a given reference signal~$y_r \in Y_r$, where
\begin{equation} \label{eq:ReferenceSignal}
        Y_r := C^1(\R_{\ge 0};\R^m) \cap W^{1,\infty}(\R_{\ge 0};\R^m),
\end{equation}
within prescribed time-varying error bounds.
To be more precise, the objective is to guarantee that
\begin{equation} \label{eq:ControlObjective}
   \forall \, t \ge 0:\  \| y(t) - y_r(t) \| \le \psi(t) 
\end{equation}
for a positive function~$\psi \in \Psi$, 
where
\begin{equation} \label{Def:FunnelBoundary}
    \begin{small}
    \Psi \!\!:=\! \!\left\{  \!\phi \!\in\! C^1 (\R_{\ge 0};\!\R) \!\cap\! L^{\infty}(\R_{\ge 0};\!\R)  \!  \left|  \! \! \begin{array}{l}
             \inf_{t\ge 0} \phi(t)>0,  \\[1mm]
             \forall\, t \!\ge\! 0 \!:  |\dot \phi(t)| \!\le\! c \phi(t) \! \! \!
        \end{array} \!  \right.\right\}   .
    \end{small}
\end{equation}

Note that \eqref{Def:FunnelBoundary} allows for a variety of funnel boundaries; in particular, exponentially decaying functions are contained, although $\psi$ is not required to be monotonically decreasing.
For a funnel boundary~$\psi \in \Psi$ and a reference $y_r \in Y_r$, we define the safe set
\begin{equation} \label{eq:SafeSet}
    \begin{aligned}
        \mathcal{C} &:= \{ (t,y) \, | \,  \| y - y_r(t) \| \le \psi(t)  \} \subset \R_{\ge 0} \times \R^m, \\
        \mathrm{int} (\mathcal{C}) &:= \{ (t,y) \, | \, \| y - y_r(t) \| < \psi(t) \}, \\
        \partial\mathcal{C} &:= \mathcal{C} \setminus \mathrm{int} (\mathcal{C}).
    \end{aligned}
\end{equation}
The control objective~\eqref{eq:ControlObjective} is equivalent to $(t,y(t)) \in \mathcal{C}$ for all~$t \ge 0$. 
In the following, we provide a model-free CBF-based control law for systems~\eqref{eq:System}, i.e., one which is independent of the functions $F$, $G$, $H$ in~\eqref{eq:System} as long as they satisfy \Cref{ass:StructuralAssumptions}. Specifically, we show how to synthesize a CBF-based feedback that renders $\mathcal{C}$ in~\eqref{eq:SafeSet} forward invariant\footnote{A set $\mathcal{A}\subseteq\R_{\ge 0}\times \R^n$ is called forward invariant for a differential equation $\dot x(t) = f(t,x(t))$ with $f\in C(\mathcal{A},\R^n)$, if for any $(0,x^0)\in\mathcal{A}$ all maximal solutions $x:[0,\omega)\to\R^n$ with $x(0)=x^0$ satisfy $(t,x(t))\in \mathcal{A}$ for all $t\in[0,\omega)$.} and thus achieves~\eqref{eq:ControlObjective} using only the available output measurement $y(t)$.

\section{Main results} \label{Sec:MainResults}
In this section, we present our main results.
First, we show that the distance between the (norm of the) tracking error and the funnel boundary defines a CBF. 
Based on this CBF, we then define a set of feedback laws and show that it renders the set~$\mathcal{C}$ in~\eqref{eq:SafeSet} forward invariant without assuming knowledge of the system data~$F,G$ or measurement of the full state~$x(t)$.
Before doing so, we briefly recap the notions of control barrier functions, safe sets, and corresponding controls.

\subsection{Control barrier functions}

Following~\cite[Def.~2]{ames2019control}, we recall the following concepts.
\begin{defn} \label{Def:TVCBF}
   Let~$U \subseteq \R^m$ and~$\mathcal{D} \subset \R_{\ge 0} \times \R^n$.
     A continuously differentiable function $b: \mathcal{D} \to \R$ is a \textit{(time-varying) CBF} for system~\eqref{eq:System}, if there exists 
     $\alpha \in \mathcal{K}_\infty^e$ such that, for all $(t,x) \in \mathcal{D}$,
    \begin{equation*}
        \sup_{u \in U} \left[ (L_F b)(t,x) \!+\! (L_G b)(t,x) u \!+\! \partial_t b(t,x) \right] > - \alpha(b(t,x)).
    \end{equation*}
    The set 
    $   \{ (t,x)\in\mathcal{D} \, | \, b(t,x) \ge 0\} $
    is called \textit{safe set} for~\eqref{eq:System} with respect to~$b$.
    For $(t,x) \in \mathcal{D}$ and a fixed $\alpha \in \mathcal{K}_\infty^e$ the set of \textit{CBF-based controls} $K_{\rm CBF}(t,x,\alpha)$ for~\eqref{eq:System} is given by 
\begin{equation*} 
    \left\{ u\! \in\! U |  (L_F b)(t,x) \!+\! (L_G b)(t,x)u \!+\! \partial_t b(t,x)  \!\ge\! - \alpha ( b(t,x)) \right\}\!.
\end{equation*}%
\end{defn}
The following proposition characterizes the safety properties of CBF-based controllers and is based on~\cite[Thm.~2]{ames2019control}.\footnote{%
The strict inequality in Def.~\ref{Def:TVCBF} ensures $\big(\partial_t b(t,x), \nabla_x b(t,x)\big)\neq 0$ for all $(t,x)\in\mathcal{D}$ with $b(t,x)=0$ and that the set $K_{\mathrm{CBF}}(t,x,\alpha)$ is non-empty for all $(t,x)\in\mathcal{D}$, see also~\cite[Rem.~1]{cohen2024safety}.}
\begin{prop}
\label{prop:CBF}
For $U \subseteq \R^m$ and~$\mathcal{D} \subset \R_{\ge 0} \times \R^n$, consider a CBF $b$ according to \Cref{Def:TVCBF} with corresponding  ${\alpha \in \mathcal{K}_\infty^e}$. 
Then any Lipschitz continuous policy $\mu:\mathcal{D}\to\R^m$ satisfying $\mu(t,x)\in K_{\mathrm{CBF}}(t,x,\alpha)$ for all $(t,x)\in \mathcal{D}$ renders the safe set $\{ (t,x)\in\mathcal{D} \, | \, b(t,x) \ge 0\}$ positively invariant with the dynamics~\eqref{eq:System} for $u(t) = \mu(t,x(t))$. 
\end{prop}

This result provides flexibility in determining the control law using QP-based implementations with the constraint set $K_{\mathrm{CBF}}$, while optimizing for different performance criteria~\cite{ames2019control}.  
However, implementation requires evaluation of the set $K_{\mathrm{CBF}}(t,x,\alpha)$, which 
requires knowledge of~$F,G$ and access to the full state~$x(t)$. Both of these issues are circumvented by the proposed method.

\subsection{Main result}
We present the main results of this article, namely \Cref{prop:bFCisCBF} and \Cref{Thm:UinK}.
\begin{prop} \label{prop:bFCisCBF}
Let~$U = \R^m$. Let a boundary~${\psi \in \Psi}$ according to~\eqref{Def:FunnelBoundary}
    and a reference~$y_r \in Y_r$ according to~\eqref{eq:ReferenceSignal} be given.
For system~\eqref{eq:System} satisfying \Cref{ass:StructuralAssumptions} the function
\begin{equation} \label{eq:CBF}
    b: \R_{\ge 0} \!\times\! \R^n \!\to\! \R, \,(t,x) \!\mapsto\! \tfrac{1}{2} \! \left(\psi(t)^2 \!-\! \|H(x)\!-\!y_r(t)\|^2 \right)
\end{equation}
is a CBF according to \Cref{Def:TVCBF} for $\mathcal{D} = \R_{\ge 0} \times  \R^n$.
\end{prop}
\begin{proof}
    We note that~$b$ is continuously differentiable on~$\mathcal{D}$.
    For the sake of better legibility, we omit the arguments of all functions whenever it is clear from the context. 
    We use the notation $e=H(x) -y_r$ and the diffeomorphism~$T$ as well as the functions in~\eqref{eq:System_InputOutput} from \Cref{ass:StructuralAssumptions}.
    Inserting the input $u=(g\circ T)^{-1}(x) \cdot (-\max\{2c,1\}e-(f\circ T)(x) +\dot y_r)$, which we are allowed to select with $c>0$ from~\eqref{Def:FunnelBoundary}, we obtain
    \begin{equation*}
        \begin{aligned}
            & \partial_t b 
            + L_Fb
            + (L_Gb)
            u 
            = \dot \psi \psi - \langle e, -\dot y_r + f\circ T + (g\circ T) u \rangle \\
            &=  \dot \psi \psi - \langle e, -\dot y_r + f\circ T \rangle \\
            & \quad - \langle e, (g\circ T)(g\circ T)^{-1}(-\max\{2c,1\}e-f\circ T+\dot y_r) \rangle \\
           & \ge \!- |\dot \psi| \psi + \max\{2c,1\}\|e\|^2 \ge - c\psi^2 + \max\{2c,1\}\|e\|^2 \\
           & \ge \!-2c\psi^2 + c \inf_{t\ge 0}\psi(t)^2 + \max\{2c,1\}\|e\|^2 \\
           & \ge \!-\!\max\{2c,1\} (\psi^2 \!-\! \| H(x)\!-\!y_r\|^2) \!+\! c\inf_{t\ge 0}\psi(t)^2 > \!- \alpha(b), 
        \end{aligned}
    \end{equation*}
    where we used~$|\dot \psi(t)| \le c \psi(t)$ for all~$t \ge 0$, and defined $\alpha(s) := 2\max\{2c,1\} s$, by which $\alpha\in\mathcal{K}_\infty^e$. 
    The supremum in \Cref{Def:TVCBF} therefore also fulfills the strict inequality.
\end{proof}
The CBF~$b$ depends only on the measured output~$y$ and not on the full state~$x$. To emphasize this, with some abuse of notation,  we denote $b(t,y)=\tfrac{1}{2}\left(\psi(t)^2-\|y-y_r(t)\|^2\right)$ in the following. 
To derive a feedback control law from the CBF~\eqref{eq:CBF}, we define the following set of candidate controls
\begin{equation} \label{eq:u_FC_from_b}
        \mathcal{U}(t,y)= \left\{ \left.  \! \frac{k \nabla_y b(t,y)}{b(t,y)}  \right| k \in[\underline{k},\bar k] \right\}, \quad (t,y) \in \text{int}(\mathcal{C}),
\end{equation}
with user-chosen constants $\overline{k}\geq \underline{k}>0$, where we recall that~$b(t,y) > 0$ for~$(t,y) \in \text{int}(\mathcal{C})$. 
Note that no system data is used to define the set~$\mathcal{U}(t,y)$.
The feedback given by~\eqref{eq:u_FC_from_b} resembles a funnel controller, cf.~\cite[Eq.~(2)]{IlchRyan02b}. 
The following result shows that control laws based on the model-free set $\mathcal{U}(t,y)$ can be interpreted as CBF-based controls.
\begin{thm} \label{Thm:UinK}
Consider a system~\eqref{eq:System} satisfying \cref{ass:StructuralAssumptions} with diffeomorphism~$T$, the CBF~$b$ from~\eqref{eq:CBF} with funnel boundary 
$\psi \in \Psi$ and reference~$y_r \in Y_r$,
$U=\R^m$ and $\mathcal{D}= \R_{\ge 0}\times\R^n$, 
the candidate control set defined in~\eqref{eq:u_FC_from_b}, and the safe set $\mathcal{C}$ defined in~\eqref{eq:SafeSet}. 
For any $\bar{q}>0$, 
there exists a linear function $\alpha\in\mathcal{K}^e_\infty$ such that, for all $(t,x)\in\R_{\ge 0}\times\R^n$, $(y,\eta)=T(x)$, we have
\[
    \Big( (t,y)\in \mathrm{int}(\mathcal{C})\ \wedge\ \|\eta\| \le \bar{q} \Big)
    \implies  \mathcal{U}(t,y) \subseteq K_{\rm CBF}(t,x,\alpha).
\]
\end{thm}
\begin{proof}
    Like before, we omit the arguments of functions whenever it is clear from the context. We use the functions in~\eqref{eq:System_InputOutput} from \Cref{ass:StructuralAssumptions} and define the constants
    \begin{align*}
        &\bar y_r := \|y_r\|_\infty, \ \hat y_r := \| \dot y_r\|_\infty, \ \bar \psi := \|\psi\|_\infty, \underline{\psi} := \inf_{s \ge 0} \psi(s),\\
        &\bar f := \max_{ \|z\| \le \bar \psi + \bar y_r, \, \|q\| \le \bar q } \| f(z,q) \| + \hat y_r, \\ 
        & \underline{g} := \min_{\|e\|=1} \min_{ \|z\| \le \bar \psi + \bar y_r, \, \|q\| \le \bar q } \langle e, g(z,q) e \rangle.
    \end{align*}
    Note that these constants are well-defined due to \Cref{ass:StructuralAssumptions} and continuity of the involved functions.
    With a similar calculation as in the proof of \Cref{prop:CBF}, we obtain, for $(t,x)\in\R_{\ge 0}\times\R^n$, $(y,\eta)=T(x)$, such that $(t,y) \in \mathrm{int}(\mathcal{C})$, and $u\in \mathcal{U}(t,y) = \mathcal{U}(t,H(x))$,
    \begin{equation*}
        \begin{aligned}
            & \partial_t b + L_fb + (L_gb) u 
            = \dot \psi \psi - \langle e, -\dot y_r + f\circ T + (g\circ T) u \rangle \\
             &= \dot \psi \psi - \langle e, f\circ T - \dot y_r \rangle +  \frac{2}{\psi^2 - \|e\|^2} \langle e, (g\circ T) k e \rangle \\
             & \ge -|\dot \psi|\psi - \psi \bar f + 2 \underline{k}\, \underline{g}\frac{\|e\|^2}{\psi^2 - \|e\|^2}  \\
             & = -|\dot \psi|\psi - \psi \bar f + \underline{k}\, \underline{g}\frac{\psi^2 \!-\! 2 b}{b}  
              = \!-|\dot \psi|\psi \!-\! \psi \bar f \!-\! 2 \underline{k}\, \underline{g} + \underline{k}\, \underline{g} \frac{\psi^2 }{b} \\
             & \ge -c \bar \psi^2 - \bar \psi \bar f - 2 \underline{k} \,\underline{g} + \frac{\underline{k}\, \underline{g} \,\underline{\psi}^2}{b} .
        \end{aligned}
    \end{equation*}
    In the first estimate, we use~$\|e\| < \psi$ since $(t,y) \in \textrm{int}(\mathcal{C})$.
    In the last estimate, we use~$|\dot \psi| \le c \psi$ according to \eqref{Def:FunnelBoundary}.
    The last line is a function of the form $l(s) = a/s - M$, for $s=b > 0$, with appropriate~$a,M>0$. 
    For $s_0=\frac{2a}{M}$, we have $l(s_0) = -\frac{M}{2}$ and $l'(s_0) = -\frac{M^2}{4a} < 0$, thus the tangent of~$l$ at~$s_0$ is given by $\tau(s) = l'(s_0) (s-s_0) + l(s_0) = -\frac{M^2}{4a}s$ and, since $l$ is convex, we have $l(s) \ge \tau(s)$ for all $s>0$. Since $b(t,y)>0$ for all $(t,y) \in \textrm{int}(\mathcal{C})$, it thus follows that
    $
        \partial_t b + L_fb + (L_gb) u \ge \tau(b).
    $
    Setting $\alpha(s)= -\tau(s)$, we find that
    $\alpha\in\mathcal{K}_\infty^e$ and $u\in K_{\rm CBF}(t,x,\alpha)$.
    Therefore, we have $\mathcal{U}(t,y) \subseteq K_{\rm CBF}(t,x,\alpha)$. 
\end{proof}
We emphasize that no knowledge about the system data from~\eqref{eq:System}, respectively~\eqref{eq:System_InputOutput}, is used in the design of feedback laws~$\mathcal{U}(t,y)$ in~\eqref{eq:u_FC_from_b}. 
The following result characterizes the closed-loop properties when choosing inputs $u$ from $\mathcal{U}(t,y)$.
In particular, we show that controls contained in $\mathcal{U}(t,y)$ render the set~$\text{int}(\mathcal{C})$ forward invariant, that is, we show $\|y(t) - y_r(t)\| < \psi(t)$ for~$t\ge0$ in the closed-loop system.
\begin{cor}
\label{cor_funnel}
Consider a system~\eqref{eq:System} satisfying \Cref{ass:StructuralAssumptions}.
    Let a funnel boundary~$\psi \in \Psi$ according to~\eqref{Def:FunnelBoundary}
    and a reference~$y_r \in Y_r$ according to~\eqref{eq:ReferenceSignal} be given.
    Let  $(0,y^0) \in \textrm{int}(\mathcal{C})$ from~\eqref{eq:SafeSet}, $\eta^0\in\R^{n-m}$, 
    $T:\R^n\to\R^n$ be the diffeomorphism from \Cref{ass:StructuralAssumptions}. Consider a feedback policy $\mu:\R_{\geq0}\times\R^m\to\R^m$ that is measurable in the first argument, continuous in the second, and satisfies ${\mu(t,y)\in\mathcal{U}(t,y)}$ for all $(t,y)\in\mathrm{int}(\mathcal{C})$. 
Then, system~\eqref{eq:System} with initial condition $x^0 = T^{-1}(y^0,\eta^0)$  
and input $u(t):=\mu(t,y(t))$
    has a solution~$x$,
    every solution can be maximally extended,
    and every maximal solution is global, i.e., $x: [0,\infty) \to \R^n$.
    Moreover,
    \begin{itemize}
        \item[(i)] $u\in L^\infty(\R_{\ge 0},\R^m)$,
        \item[(ii)] there exists~$\varepsilon \in (0,1)$ such that, for all~$t \ge 0$, we have $b(t,y(t)) \ge \tfrac12 (1-\varepsilon^2)\psi(t)^2$.
    \end{itemize}
    In particular, the latter implies $(t,y(t))\in \mathrm{int}(\mathcal{C})$ for~$t \ge 0$.
\end{cor}
\begin{proof}
    The existence of a local maximal solution ${x: [0,\omega) \to \R^n}$, $\omega \in (0,\infty]$, follows from~\cite[Thm.~5]{IlchRyan02b}.
    We show statement~(ii) with~$\varepsilon \in (0,1)$ determined below. 
    Note that for any $\zeta \in C(\R_{\ge 0},\R^m)$ with $\|\zeta\|_\infty \le \bar \psi + \bar y_r$ it follows from \Cref{ass:StructuralAssumptions} that the global solution $\eta(\cdot;\eta^0,\zeta)$ exists. Then, again by \Cref{ass:StructuralAssumptions}, there exists $\bar q>0$ such that  $\|\eta(\cdot;\eta^0,\zeta)\|_\infty \le \bar q$ for any such function $\zeta$. 
    We set
    \begin{align*}
               &\hat \varepsilon \in (0,1) \ \text{s.t.} \ \frac{\hat \varepsilon^2}{1- \hat \varepsilon^2} \ge \frac{ \left\|\tfrac{\dot \psi}{\psi}\right\|_\infty + \sup_{s \ge 0}\tfrac{1}{\psi(s)}(\bar f + \hat y_r) }{2\underline{g} \underline{k} \inf_{s \ge 0} \tfrac{1}{\psi(s)}}, \\
        & \varepsilon = \max \{ \|y(0) - y_r(0)\|/\psi(0), \hat \varepsilon \} < 1.
    \end{align*}
    In the following, we use $e(t)=y(t)-y_r(t)$.
    Seeking a contradiction, we assume that there exists~$t^* \in [0,\omega)$ such that $\psi(t^*) > \|e(t^*)\| > \varepsilon \psi(t^*)$.
    Note that $\|e(0)\| \le \varepsilon \psi(0)$ by construction of~$\varepsilon$.
    Due to continuity of the involved signals, there exists $t_* := \max\{ t \in [0,t^*) \, | \, \|e(t)\| = \varepsilon \psi(t) \}$. Thus, $\|e(t)\|\ge \varepsilon \psi(t)$ and $\tfrac{\|e\|^2}{\psi^2-\|e\|^2} \ge \tfrac{\varepsilon^2}{1-\varepsilon^2}$ for all $t\in [t_*,t^*]$.
    We omit the time argument and calculate for $t \in [t_*,t^*]$
    \begin{equation*}
        \begin{aligned}
           & \tfrac{\textrm{d}}{\textrm{d} t} \tfrac{1}{2} \| \tfrac{e}{\psi}\|^2 
            = \langle \tfrac{e}{\psi},  \tfrac{\psi \dot e - e \dot \psi}{\psi^2}  \rangle 
            = \langle  \tfrac{e}{\psi},  -\tfrac{e}{\psi} \tfrac{\dot \psi}{\psi} + \tfrac{1}{\psi} (f - \dot y_r + g\,u) \rangle \\
            & \le \left\| \tfrac{\dot \psi}{\psi} \right\|_\infty + (\bar f + \hat y_r) \sup_{s \ge 0} \tfrac{1}{\psi(s)} - \tfrac{\|e\|^2}{\psi^2-\|e\|^2} \inf_{s \ge 0} \tfrac{2\underline{k} \underline{g}}{\psi(s)}  \\
            &\le \left\| \tfrac{\dot \psi}{\psi} \right\|_\infty + (\bar f + \hat y_r) \sup_{s \ge 0} \tfrac{1}{\psi(s)} - \tfrac{\varepsilon^2}{1-\varepsilon^2} \inf_{s \ge 0} \tfrac{2\underline{k} \underline{g}}{\psi(s)}               \le 0.           
        \end{aligned}
    \end{equation*}
    The contradiction $\varepsilon < \|\tfrac{e(t^*)}{\psi(t^*)}\| \le \|\tfrac{e(t_*)}{\psi(t_*)}\| = \varepsilon$ arises upon integration. Thus, $e(t)/\psi(t) \le \varepsilon$ for all $t \in [0,\omega)$. 
    Then, the maximal solution is global, i.e., $\omega=\infty$.
    This yields the assertion $b(t,y(t)) \ge \tfrac12 (1-\varepsilon^2) \psi(t)^2$ for all $t \ge 0$.
    Assertion~(i) is then a consequence of~(ii), with ${\|u(t)\|\leq \tfrac{\bar{k}\bar{\psi}}{(1-\varepsilon^2)\underline{\psi}}}$ $\forall t\geq 0$ using~\eqref{eq:u_FC_from_b}.
\end{proof}
To obtain a controller, which allows the system to evolve outside the safe set $\mathcal{C}$ from~\eqref{eq:SafeSet}, we can consider the following modified control set with a parameter $\delta>0$,
\begin{equation} \label{eq:u_FC_saturated}
    \mathcal{U}_\delta(t,y) \!=\! \left\{ \! \left.  \! \tfrac{k \nabla_y b(t,y)}{\max\{b(t,y),\delta\}}  \right| 
          k \!\in\! [\underline{k},\overline{k}] \! \right\}\!, \ (t,y) \!\in \! \R_{\geq0} \! \times \! \R^m\!.
\end{equation}
However, to ensure asymptotic stability of $\mathcal{C}$ (cf. \cite[Prop.~2]{ames2016control}), the constant $\underline{k}/\delta > 0$ needs to be chosen sufficiently large, i.e., system knowledge has to be included in the controller design. 
Nevertheless, this sets the proposed controller design apart 
from funnel control approaches~\cite{IlchRyan02b,berger2018funnel,BergIlch21} 
as they cannot achieve the evolution of the tracking error $e$ within the funnel 
boundaries at a later point in time, if the initial error $e(0)$ does not satisfy $(0,y(0)) \in \mathrm{int}(\mathcal{C})$.

\section{Discussion} \label{Sec:Discussion}
We reflect on the application of the theoretical results to safe control of unknown dynamical systems. We first discuss how the set~$\mathcal{U}(t,y)$ can be used to define suitable controllers (\Cref{sec:discuss_opt}) and then contrast the results with existing results from CBF-based control (\Cref{sec:discuss_CBF}), funnel-based control (\Cref{sec:discuss_funnel}), and the work studying PPC as a reciprocal CBF-based control~\cite{namerikawa2024equivalence} (\Cref{Sec:PPC_and_TVRCBF}). 

\subsection{Optimization-based CBF implementation}
\label{sec:discuss_opt}
Given the set $\mathcal{U}(t,y)$ in Theorem~\ref{Thm:UinK}, we can use a CBF-inspired QP to determine the control feedback
\begin{equation}\label{eq:QP_CBF}
\begin{aligned}
\mu(t,y):=\argmin_{u\in\mathbb{R}^m}&~\mathrm{cost}(t,u,y) \ \
\text{s.t. } \ u\in\mathcal{U}(t,y).
\end{aligned}
\end{equation}
Here, $\mathrm{cost}(t,u,y)$ is a convex quadratic cost function while the set $\mathcal{U}(t,y)$ from~\eqref{eq:u_FC_from_b} can be written as a set of linear inequality constraints, i.e.,~\eqref{eq:QP_CBF} is a QP. 
The quadratic cost can encode a wide range of different performance criteria,  see~\cite{ames2019control}.
Of particular relevance is penalizing the difference to a desired input, i.e., $\mathrm{cost}(t,u,y)=\|u-u_r(t)\|^2$, which is also referred to as a minimally invasive safety filter~\cite{wabersich2023data}. In this case, Problem~\eqref{eq:QP_CBF} has an analytical solution (cf.~\cite[Eq.~(15)]{cohen2024safety}). 
We note that there exist results ensuring Lipschitz continuity of the resulting feedback $\mu(t,y)$, cf.~\cite{jankovic2018robust}. However, due to a different structure of $\mathcal{U}(t,y)$, it is not obvious if these results also apply for the proposed implementation~\eqref{eq:QP_CBF}. This requires further investigation.
\subsection{Model-free CBF-based control}
\label{sec:discuss_CBF}
The proposed approach, in particular the implementation using~\eqref{eq:QP_CBF}, shares many of the benefits of CBFs~\cite{ames2016control,ames2019control}. Specifically, safety is explicitly enforced through the set $\mathcal{U}(t,y)$ constraining the set of admissible control inputs. Thus, many secondary control objectives related to stability or other performance criteria can be flexibly included through the $\mathrm{cost}$, without affecting safety properties~\cite{wabersich2023data}. 
The key benefit of the proposed approach is its model-free nature. While standard CBF-based control implementations require access to the full state $x$ and the system data $F,G$ in~\eqref{eq:System}, Problem~\eqref{eq:QP_CBF} only utilizes the output measurement~$y$ and the structural condition in \cref{ass:StructuralAssumptions}. 
The most related research directions from the CBF literature are recent approaches using reduced-order (kinematic) models~\cite{cohen2024safety,molnar2021model}. Specifically,  \cite{molnar2021model} is also referred to as model-free due to the minimal assumptions, however, it is limited to a specific class of fully actuated mechanical systems. 

\subsection{CBF-based funnel}
\label{sec:discuss_funnel}
The control input in~\eqref{eq:u_FC_from_b} resembles a funnel controller and the invariance properties shown in Corollary~\ref{cor_funnel} are in line with existing results for funnel control~\cite{BergIlch21}. Compared to these existing results, the proposed approach has two benefits:
First, by characterizing the properties of the feedback through a set inclusion, we can use the flexible implementation strategy introduced in Section~\ref{sec:discuss_opt}, which can account for additional performance criteria without jeopardizing the theoretical guarantees. 
Furthermore, the implementation strategy~\eqref{eq:u_FC_saturated} with $\mathcal{U}_\delta(t,y)$ allows us to naturally consider the problem of smoothly returning back into a safe set $\mathcal{C}$, which is a standard property of (zeroing) CBFs~\cite{ames2016control}. 

\subsection{Prescribed performance control and reciprocal CBF} \label{Sec:PPC_and_TVRCBF}
As mentioned in \Cref{Sec:Introduction}, a connection between PPC and controls based on reciprocal CBFs has been discovered in~\cite{namerikawa2024equivalence}.
By setting the auxiliary variables $e_1(t,x) := (x-x_r(t))/\psi(t)$ and $\xi(t,x) := e_1(t,x)/\sqrt{1-\|e_1(t,x)\|^2}$, one may
define ${ B_{\rm FC} :  \R^n  \to \R_{\ge 0}}$, $\xi  \mapsto \tfrac{1}{2} \|\xi\|^2$.
Using~$B_{\rm FC}$, we can define the feedback $u(t,x) = -k \nabla_x B_{\rm FC}(\xi(t,x)) $ for some $k > 0$.
Then, thanks to the results in~\cite{namerikawa2024equivalence}, we can show that 
i) $u$ keeps the system within predefined boundaries, i.e., $\|x(t)-x_r(t)\|<\psi(t)$ (cf.~\cite[Thm.~2]{namerikawa2024equivalence}), 
ii) $B_{\rm FC}$ is a reciprocal CBF (RCBF) (cf.~\cite[Thm.~2]{namerikawa2024equivalence}),
and iii) $u$ defines a RCBF-based controller (cf.~\cite[Thm.~4]{namerikawa2024equivalence}).
In comparison, key benefits of the approach proposed in the present paper are the fact that zeroing CBFs are used, internal dynamics are considered and, most importantly, the implementation can be naturally adjusted to take  different optimization criteria into account (cf. \Cref{sec:discuss_opt}) and provide a control strategy to return to the safe set, cf.~\eqref{eq:u_FC_saturated}.

\section{Numerical simulation}
\label{Sec:num}
The following numerical simulations demonstrate that the proposed methodology, while no model parameters are used, inherits the benefits of both CBF-based and funnel-based control methods, that is, evolution in safe sets is guaranteed.\\
\textit{Setup:} 
We consider the problem of controlling an unmanned surface vessel, taken from~\cite{namerikawa2024equivalence}.
The dynamics are given by
\begin{align*}
\dot{x}=
\begin{small}
\begin{bmatrix}
-\sin(\tan^{-1}(p_y/p_x)) \\
\cos(\tan^{-1}(p_y/p_x)) \\
0
\end{bmatrix}
\end{small}
+ 
\begin{small}
\begin{bmatrix}
    \cos(\phi) & -\sin(\phi) & 0 \\
    \sin(\phi) & \cos(\phi) & 0 \\
    0&0&1
\end{bmatrix}
\end{small}
\begin{small}
\begin{pmatrix}
    v \\ w \\ r
\end{pmatrix}
\end{small}
\end{align*}
with $y=x=[p_x,p_y,\phi]$ and  $u=[v,w,r]$. These dynamics consist of the known kinematics and an unknown position-dependent drift term that may result from waves.
We have a given reference trajectory $y_r$ with a corresponding funnel size $\psi>0$ such that Assumption~\ref{ass:StructuralAssumptions} holds on the set $\mathcal{C}$.\footnote{%
The definiteness condition in Assumption~\ref{ass:StructuralAssumptions}\,a)  holds, if $|\phi|<\pi/2$. 
 The reference $y_r(t)=[-8/10\cdot t+8,4\cos(3/20\pi t),\arctan(3/2\pi\sin(3/20\pi t)]^\top$ and funnel boundary $\psi(t)=1.3\exp(-2t)+0.2$ ensure that $|\phi|<\pi/2$ holds for all $(t,y)\in\mathcal{C}$. Hence the closed-loop guarantees remain valid.}  
\\ \textit{Simulations: }
Similar to~\cite{namerikawa2024equivalence}, we implement a standard funnel controller with a fixed gain $k=1$. 
In addition, we implement a model-free CBF-based control using Theorem~\ref{Thm:UinK} with $\underline{k}=10^{-3}<\overline{k}=10^3$. 
Specifically, we used the QP implementation~\eqref{eq:QP_CBF} with a quadratic cost that minimizes the distance to a model-based input reference $u_r$.\footnote{$u_r$ is numerically computed based on the known kinematic terms in the dynamics and the derivative of the reference trajectory $\dot{y}_r$.}
The results can be seen in Figure~\ref{fig}. 
As expected, both controllers ensure invariance of the safe set without requiring any model knowledge. However, the CBF-based controller has additional degrees of freedom, which are utilized here to generate control inputs that stay close to a desired input reference. Quantitatively, the mean squared error to the desired input is reduced by~$84~\%$. 
\begin{figure}[t]
\centering 
\includegraphics[width=0.9\columnwidth]{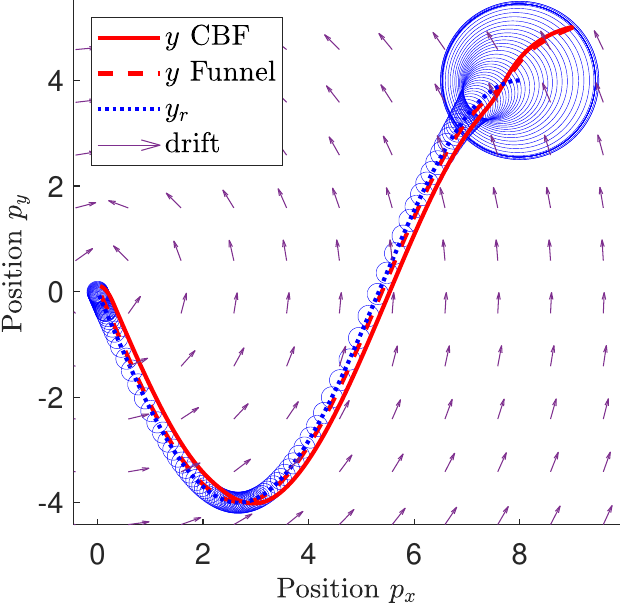}
\includegraphics[width=0.9\columnwidth]{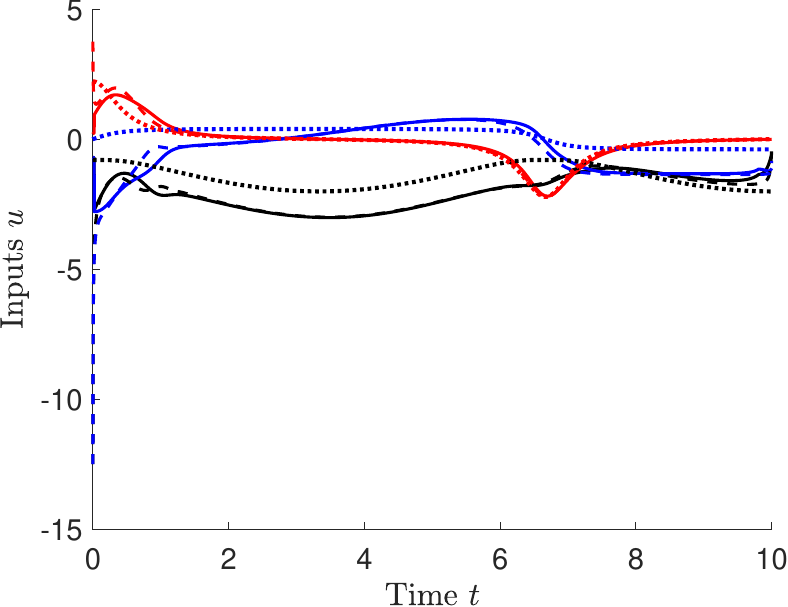}
 \caption{Top: Funnel (blue circles) around reference $y_r$ (blue, dotted) containing closed-loop trajectories of proposed CBF-based controller (red, solid) and funnel controller (red, dashed). 
Bottom: Closed-loop control input $u$ for proposed CBF-based controller (solid), FC (dashed) and input references $u_r$ (dotted) for all three input components (red,blue,black).}
 \label{fig}
\end{figure} 

\section{Conclusion}
In this work, we have presented a control barrier function (CBF) framework to address the problem of tracking reference signals within prescribed error bounds for a class of nonlinear multi-input multi-output systems with relative degree one. By integrating insights from funnel control, we have identified a class of CBF-based control laws that can be synthesized without explicit knowledge of the system dynamics. Instead, our approach relies solely on structural assumptions about the system, enabling the design of a model-free CBF-based controller that establishes safety guarantees for the tracking error.
Future research will focus on extending these results to systems with higher relative degrees~-- a critical step toward broadening the applicability of the method. 
Also, building upon~\cite{Kiss23}, we aim to generalize our model-free CBF approach to time-delay and hysteresis systems, further enhancing its utility in real-world applications.

\bibliographystyle{IEEEtran}
\bibliography{references}
\end{document}